\documentclass{article}

\usepackage{arxiv}

\usepackage[utf8]{inputenc} 
\usepackage[T1]{fontenc}    
\usepackage{hyperref}       
\usepackage{url}            
\usepackage{booktabs}       
\usepackage{amsfonts}       
\usepackage{nicefrac}       
\usepackage{microtype}      
\usepackage{lipsum}
\usepackage{amsmath}

\usepackage{amsfonts,amssymb,stmaryrd}
\usepackage{graphicx}
\usepackage{amsthm}
\usepackage{xcolor}

\newcommand{\R}{\mathbb{R}}

\newtheorem{theorem}{Theorem}[section]
\newtheorem{proposition}{Proposition}[section]
\newtheorem{lemma}{Lemma}[section]
\newtheorem{corollary}{Corollary}[section]

\newtheorem{definition}{Definition}[section]
\newtheorem{remark}{Remark}[section]
\newtheorem{example}{Example}[section]

\usepackage{subcaption}
\usepackage{cite}
\usepackage{authblk}
\newcommand{\p}{\partial}
\newcommand{\bb}{\begin{equation}}
\newcommand{\ee}{\end{equation}}
\newcommand{\ba}{\begin{array}}
\newcommand{\ea}{\end{array}}
\newcommand{\f}{\frac}
\usepackage[all]{xy}
\newcommand{\ds}{\displaystyle}

\newcommand{\al}{\alpha}
\newcommand{\be}{\beta}

\newcommand{\sign}{\text{sgn}\,}
\newcommand{\e}{\mathbb{E}}

\newcommand{\N}{{\mathbb N}}
\newcommand{\M}{{\cal M}}

\newcommand{\Gr}{\text{Gr}}
\newcommand{\supp}{\text{supp}}

\newcommand{\letra}{ \renewcommand{\labelenumi}{\alph{enumi})}}

\numberwithin{equation}{section}

\usepackage{ulem}

\usepackage[most]{tcolorbox}

\definecolor{TextColor}{rgb}{0.75, 0.75, 0.75}

\title{Intrinsic geometry and wave-breaking phenomena in solutions of the Camassa-Holm equation}

\author[1] {Igor Leite Freire}
\affil[1]{Departamento de Matemática, Universidade Federal de São Carlos\\
Rodovia Washington Luís, Km 235, 13565-905\\
São Carlos-SP, Brasil\\
  \texttt{igor.freire@ufscar.br} \\
  \texttt{igor.leite.freire@gmail.com}}
 
\begin{document}
\maketitle
\begin{abstract}
Pseudospherical surfaces determined by Cauchy problems involving the Camassa-Holm equation are considered herein. We study how global solutions influence the corresponding surface, as well as we investigate two sorts of singularities of the metric: the first one is just when the co-frame of dual form is not linearly independent. The second sort of singularity is that arising from solutions blowing up. In particular, it is shown that the metric blows up if and only if the solution breaks in finite time.
\end{abstract}

{\bf MSC classification 2020:} 35A01, 58J60, 37K40, 35Q51.

\keywords{Equations describing pseudospherical surfaces \and Geometric analysis \and Existence of metrics \and Blow up of metrics}
\newpage
\tableofcontents
\newpage

\section{Introduction}\label{sec1}

The Camassa-Holm (CH) equation
\bb\label{1.0.1}
u_t-u_{txx}+3uu_x=2u_xu_{xx}+uu_{xxx},
\ee
originally derived as a model for shallow water wave dynamics, has since become a cornerstone in the study of non-linear phenomena, integrable systems, and  analysis of PDEs. One of its remarkable features is its connection to pseudospherical surfaces (PSS), enabling a deep interplay between differential equations, geometry, and physical applications.

Despite extensive studies on the CH equation's integrability and wave-breaking phenomena, its geometric consequences have been barely explored. In a recent paper \cite{freire-ch} the author studied the geometry of PSS determined by Cauchy problems involving the CH equation.

The start point for \cite{freire-ch} is Reyes' work \cite{reyes2002}, where the geometric integrability of the CH equation was first established. The results reported in \cite{freire-ch} can be summarised as follows:
\begin{itemize}
    \item the non-local form of the CH equation, or its integral formulation, is geometrically integrable;
    \item any non-trivial initial datum determines a metric for a PSS;
    \item solutions $u(x,t)$ emanating from odd initial data satisfying $m_0(x)=u(x,0)-u_{xx}(x,0)\leq 0$, for $x\geq 0$, define a first fundamental blowing up within a finite region.
\end{itemize}

Although \cite{freire-ch} shed some light on qualitative aspects of surfaces determined by the CH equation, many other questions remained open, such as:
\begin{itemize}
    \item It was show that local solutions may define abstract surfaces provided that the metric is defined on subsets of strips determined by the initial datum. What would happen when global solutions are considered?
    \item In \cite{freire-ch} was considered a specific scenario for which the solution breaks at finite time. From a geometric perspective, that wave-breaking solution leads to a  metric tensor that becomes unbounded within a finite region. What might be said about metrics defined by other wave-breaking solutions? Does any solution breaking in finite time lead to a metric blowing up?
    \item Some qualitative results in the literature of the CH equation enable us to describe qualitatively the solution. Can we qualitatively describe the corresponding metric for these solutions? If yes, what can be said? For a negative answer, why not?
\end{itemize}

This paper is concerned to answer the questions above.  The results reported here contribute to the broader understanding of non-linear systems, where the interplay between geometry and physical phenomena described by the model offers new perspectives from the point of view of geometric analysis, including wave propagation and the onset of singularities in integrable models. More concretely, our contribution are:
\begin{itemize}
    \item Establishment of direct correspondence between wave-breaking phenomena and metric blow-up in the CH equation. Our main result concerning this topic is that any solution developing wave-breaking leads to a blowing up metric, see Theorem \ref{teo6.1}.
    \item Extension of previous of local to global nature, exploring their geometric implications.
    \item Discussion of examples highlighting the practical significance of the findings.
\end{itemize}

The outline of the paper is as follows: In section \ref{sec2} we revisit some basic and relevant aspects of the CH equation, with main focus on the two-dimensional Riemannian geometry determined by its solutions and open problems regarding its geometric analysis. Next, in section \ref{sec3} we fix the notation used throughout the manuscript, recall basic notions and state our main results. In section \ref{sec4} we recall qualitative results regarding the CH equation, that are widely employed in section \ref{sec5}, where our main results are proved. In section \ref{sec6} we show that the metric of the surface blows up if and only if the solution breaks in finite time. Some examples illustrating our main results are discussed in section \ref{sec7}, while our discussions and conclusions are presented in sections \ref{sec8} and \ref{sec9}, respectively.

\section{Geometric aspects of the CH equation}\label{sec2}

Despite being primarily deduced as an approximation for the description of waves propagating in shallow water regimes, the CH equation proved to have several interesting properties related to integrability \cite{chprl}. If we denote $$m(x,t):=u(x,t)-u_{xx}(x,t),$$
which is known as momentum \cite{chprl}, then \eqref{1.0.1} can be rewritten as an evolution equation for $m$, namely,
\bb\label{2.0.1}
m_t+2u_xm+um_x=0.
\ee

Equation \eqref{2.0.1} has a bi-Hamiltonian structure \cite{chprl}. In particular, the functional
\bb\label{2.0.2}
{\cal H}_1=\f{1}{2}\int_\R (u^2+u_x^ 2)dx,
\ee
plays vital importance not only because it is a Hamiltonian, but also because it is an invariant for zero background solutions of the CH equation.

As a consequence of its bi-Hamiltonian structure, \eqref{2.0.1} has also a recursion operator and infinitely many symmetries as well, being also integrable in this sense. The reader is referred to \cite[Chapter 7]{olverbook} or \cite{olverjmp} for further details about recursion operators and integrability. 

It is still worth of mention that Camassa and Holm showed a Lax formulation \cite{chprl} for \eqref{1.0.1}
\bb\label{2.0.3}
\psi_{xx}=\ds{\Big(\f{1}{4}-\f{m}{2\lambda}\Big)\psi},\,\,
\,\,
\psi_t=\ds{-(\lambda+u)\psi_x+\f{1}{2}u_x\psi}
\ee
as well as continuous, piecewise soliton like solutions, called peakons. For a review on the CH and related equations, see \cite{freire-cm}.

The CH equation, or its solutions, can also be studied from geometric perspectives \cite{const2000-1,const2002-jpa,const-jnl,reyes2002}. We shall briefly discuss \cite{const2000-1,reyes2002} which are the main inspirations for this paper, the first being concerned with infinite dimensional Riemmanian geometry, whereas the latter is concerned with an abstract two-dimensional Riemannian manifold, whose importance for this paper is crucial.

Equation \eqref{1.0.1} can be associated with the geometric flow in an infinite dimensional manifold ${\cal D}^3(\R)$ modelled by a Hilbert space in which we can endow a (weak) Riemannian metric \cite{const2000-1}. The geodesics in ${\cal D}^3(\R)$ can either exist globally \cite[Theorem 6.1]{const2000-1} or breakdown in finite time \cite[Theorems 6.3 and 6.4]{const2000-1} and, in particular, geodesics starting, at the identity, with initial velocity corresponding to initial datum leading to breaking solutions will also develop singularities at finite time \cite[Theorem 6.3]{const2000-1}.

A different geometric perspective for the CH equation was given by Reyes \cite{reyes2002}, who showed it describes pseudospherical surfaces \cite[Theorem 1]{reyes2002} {\it à la} Chern and Tenenblat \cite{chern}, e.g. see \cite[Definition 2.1]{keti2015}. 

\begin{definition}\label{def2.1}
    A pseudospherical surface $(PSS)$ is a two-dimensional Riemannian manifold whose Gaussian curvature is constant and negative. 
\end{definition}

For now it suffices saying that an equation describes pseudospherical surfaces, or is of the pseudospherical type, henceforth referred as PSS equation, when the equation is the compatibility condition of the structure equations
\bb\label{2.2.1}
d\omega_1=\omega_3\wedge\omega_2,\quad d\omega_2=\omega_1\wedge\omega_3,\quad d\omega_3=-{\cal K}\omega_1\wedge\omega_2,
\ee
for a PSS.

In his work Reyes showed that if $u$ is a solution of the CH equation, $m$ is its corresponding momentum, then the one-forms
\bb\label{2.2.2}
\ba{lcl}
\omega_1&=&\ds{\Big(\f{\lambda}{2}+\f{1}{2\lambda}-m\Big)dx+\Big(um+\f{\lambda}{2}u-\f{u}{2\lambda}-\f{1}{2}-\f{\lambda^2}{2}\Big)dt},\\
\\
\omega_2&=&-u_xdt,\\
\\
\omega_3&=&\ds{\Big(m+\f{1}{2\lambda}-\f{\lambda}{2}\Big)dx+\Big(\f{\lambda^2}{2}-\f{1}{2}-\f{u}{2\lambda}-\f{\lambda}{2}u-um\Big)dt},
\ea
\ee
satisfy \eqref{2.2.1}, for any $\lambda\in\R\setminus\{0\}$ and ${\cal K}=-1$. This implies that the domain of the solution $u$, under certain circumstances, can be endowed with a Riemannian metric $g=\omega_1^2+\omega_2^2$ of a PSS, also known as first fundamental form of the surface. From \eqref{2.2.2}, the corresponding metric is
\bb\label{2.2.3}
\ba{lcl}
g&=&\ds{\Big(\f{\lambda}{2}+\f{1}{2\lambda}-m\Big)^2dx^2}+\ds{2\Big(\f{\lambda}{2}+\f{1}{2\lambda}-m\Big)\Big(um+\f{\lambda}{2}u-\f{u}{2\lambda}-\f{1}{2}-\f{\lambda^2}{2}\Big)dxdt}\\
\\
&+&\ds{\Big[u_x^2+\Big(um+\f{\lambda}{2}u-\f{u}{2\lambda}-\f{1}{2}-\f{\lambda^2}{2}\Big)^2}\Big]dt^2=:\ds{g_{11}dx^2+2g_{12}dxdt+g_{22}dt^2}.
\ea
\ee

More precisely, the work of Reyes showed that, in fact, the Camassa-Holm equation is geometrically integrable, in the sense that its solutions may describe a one-parameter family of non-trivial pseudospherical surfaces \cite[Corollary 1]{reyes2002}. This is a consequence of the fact that the parameter $\lambda$ in \eqref{2.2.2} cannot be removed under a gauge transformation\footnote{A PSS equation can be defined by more than one choice of forms \cite{sasaki}. Even for the CH equation, our triad \eqref{2.2.2} is obtained by a gauge transformation from the original forms discovered by Reyes \cite[Theorem 1]{reyes2002}, see \cite[Remark 6]{reyes2002}.}

A smooth solution of a PSS equation leads to smooth one-forms $\omega_1,\omega_2,\omega_3$ and then the corresponding first fundamental form will inherit the same regularity. The solutions considered by Constantin \cite{const2000-1}, in contrast, are not necessarily $C^\infty$, showing an enormous difference between \cite{const2000-1} and \cite{tarcisio,keti2015,reyes2000,reyes2002,reyes2006-sel,reyes2006-jde,reyes2011,sasaki} in terms of the regularity of the objects considered.
 
\section{Notation, notions and main results}\label{sec3}

Throughout this paper $u=u(x,t)$ denotes a function depending on the variables $x$ and $t$, whose physical meaning, when considering the model \eqref{1.0.1}, are height of the free surface of water above a flat bottom, space and time, respectively. From a geometric point of view, $x$ and $t$ are coordinates of a domain in $\R^2$ in which the function $u$ is defined. We denote by $u(x,\cdot)$ and $u(\cdot,t)$ the functions $t\mapsto u(x,t)$, for fixed $x$, and $x\mapsto u(x,t)$, for fixed $t$, respectively.

For given two non-empty and connected subsets $I,J\subseteq\R$, the notation $u\in C^0(I\times J)$ means that $u=u(x,t)$ is continuous with respect to both variables in $I\times J$. By $u_x$ or $\p_x u$ we denote partial derivative of $u$ with respect to its first argument, while similarly $u_t$ or $\p_t u$ will denote partial derivative with respect to the second argument. We can also consider higher order derivatives using similar convention.

The set of ordered $n-th$ derivatives of $u$, $n\in\N$, is denoted by $u_{(n)}$. By convention, $u_{(0)}=u$. Whenever $u$ and its all derivatives up to order $k\in\N\cup\{0\}$ are continuous on the domain of $u$, we then write $u\in C^k$. The sets of smooth functions defined on a domain $\Omega\subseteq\R^2$ is denoted by $C^\infty(\Omega)$. 

Given $n\in\N$, a non-empty set $I\subseteq\R$ and a Banach space $X$, we say that $u\in C^n(X,I)$ whenever $\p_x^ku(\cdot,t) \in C^0(X,I)$, $0\leq k\leq n$. Moreover, $u\in C^0(X,I)$ means $u(\cdot,t)\in X$ and $\|u\|_{C^0}=\sup_{t\in I}\|u(\cdot,t)\|_X$.

\subsection{Sobolev spaces}

We assume familiarity with Sobolev spaces and Fourier transform. We give a concise presentation in order not to increase the manuscript. The author less familiar with these concepts is guided to \cite{freire-ch}, where a short revision on these spaces is presented in a similar context.

We denote by $\langle \cdot,\cdot\rangle_s$ and $\|\cdot\|_s$, $s\in\R$, the inner product in $H^s(\R)$ and its induced norm, respectively, whereas by $\|\cdot\|_{L^p(\R)}$ we denote the norm in the $L^p(\R)$ space, for finite $p$, and $\|\cdot\|_\infty$ otherwise. In particular, ${\cal S}(\R)\subset H^s(\R)\subset H^t(\R) \subset {\cal S}'(\R)$, for any $s\geq t$.

The following is a cornerstone result for our developments.

\begin{lemma}\label{lema3.1}{\tt(Sobolev Embedding Theorem, \cite[Proposition 1.2, page 317]{taylor})}
If $s>1/2$, then each $u\in H^s(\R)$ is bounded and continuous. In addition, if $s>1/2+k$, $k\in\N$, then $H^s(\R)\subseteq C^k(\R)\cap L^\infty(\R)$.
\end{lemma}

As we will soon see, the natural Sobolev space for our purposes is precisely $H^4(\R)$, which, in view of the precedent result, is embedded into $ C^3(\R)\cap L^\infty(\R)$.

Formally, if $m_0(x)=\Lambda^2(u_0)=u_0(x)-u_0''(x)$ then
$$
u_0(x)=(\Lambda^{-2}m_0)(y)=\f{1}{2}\int_\R e^{-|x-y|}m_0(y)dy.
$$

Another frequent operator seen in this paper is 
\bb\label{3.1.4}
\p_x\Lambda^{-2}=(\p_xg)(x)=-\f{\sign(x)}{2}e^{-|x|},
\ee
that acts on $f$ through the formula $(\p_x\Lambda^{-2}(f))(x)=-\f{1}{2}(\sign(\cdot)e^{-|\cdot|}\ast f(\cdot))(x)$.

\subsection{Intrinsic geometry and PSS}

Let $\e$ be the usual three-dimensional euclidean space, with canonical inner product $\langle\cdot,\cdot\rangle$ and $\M\subseteq\e$ be an open, non-empty set, which we shall henceforth identify with a surface. A one-form $\omega=f(x,t)dx+g(x,t)dt$ defined on $\M$ is said to be of class $C^k$ if and only if its coefficients $f$ and $g$ are $C^k$ functions. 

We say that a triad of $C^k$ one forms $\{\omega_1,\omega_2,\omega_3\}$ endows $\M$ with a PSS structure with Gaussian curvature ${\cal K}=-1$, if $\{\omega_1,\omega_2\}$ is linearly independent, that is expressed through the condition $\omega_1\wedge\omega_2\big|_\M\neq0$, and the following equations
\bb\label{3.2.1}
d\omega_1=\omega_3\wedge\omega_2,\quad d\omega_2=\omega_1\wedge\omega_3,\quad d\omega_3=\omega_1\wedge\omega_2
\ee
are satisfied.

The form $\omega_3$ is called {\it Levi-Civita connection} and it is completely determined by the other two one-forms \cite[Lemma 5.1, page 289]{neil}, as well as the Gaussian curvature of $\M$ \cite[Theorem 2.1, page 329]{neil}. Since the forms $\omega_1,\omega_2$, for each point $p\in\M$, are dual elements of the basis of the corresponding tangent space, then they are intrinsic objects associated to the surface, as well as any other geometry object described only by them.

\begin{definition}\label{def3.1}
Let $\omega_1$ and $\omega_2$ be given one-forms on a surface $\M$ in $\e$, such that $\{\omega_1,\omega_2\}$ is LI, and $p\in\M$. The first fundamental form of $\M$ is defined, on each each tangent space $T_p\M$ and for any $v\in T_p\M$, by $I(v)=\omega_1(v)^2+\omega_2(v)^2$.
\end{definition}

Using the convention $\al\be=\al\otimes\be$ and $\al^2=\al\al$, for any one-forms $\al$ and $\be$, we can rewrite the first fundamental form as
\bb\label{3.2.7}
I=\omega_1^2+\omega_2^2.
\ee 

\subsection{Main results}\label{sub2.3}

Let us consider the CH equation \eqref{1.0.1} and
\bb\label{3.3.1}
{\cal E}[u]:=u_t-u_{txx}+3uu_x-2u_xu_{xx}+uu_{xxx},\quad \overline{\cal E}[u]:=u_t+uu_x+\p_x\Lambda^{-2}\Big(u^2+\f{u_x^2}{2}\Big).
\ee

While ${\cal E}[u]$ is a well defined quantity for smooth functions $u=u(x,t)$,
the same cannot be said for $\overline{{\cal E}}[u]$. Its expression in \eqref{3.3.1} has to be seen at a formal level, in the sense it may be meaningless depending on where $u$ belogs to. However, if we restrict ourselves to functions $u\in C^0(H^4(\R),[0,T))\cap C^1(H^3(\R),[0,T))$, then we have the identities
$$
{\cal E}[u]=\Lambda^{2}(\overline{{\cal E}}[u]),
$$
that, in other words, reads to
\bb\label{3.3.2}
u_t-u_{txx}+3uu_x-2u_xu_{xx}+uu_{xxx}=(1-\p_x^2)\Big(u_t+uu_x+\p_x\Lambda^{-2}\Big(u^2+\f{u_x^2}{2}\Big)\Big).
\ee

Suppose that $u$ is a solution of the CH equation \eqref{1.0.1}. Then $u$ is a solution of the non-local (first order) evolution equation
\bb\label{3.3.3}
u_t+uu_x+\p_x\Lambda^{-2}\Big(u^2+\f{u_x^2}{2}\Big)=0.
\ee

Conversely, assuming that $u\in C^0(H^4(\R),[0,T))\cap C^1(H^3(\R),[0,T))$ is a solution of \eqref{3.3.3}, then \eqref{3.3.2} tells us that $u$ is a solution of \eqref{1.0.1}.

The above examples show that a solution for \eqref{3.3.3} is not necessarily a solution of the \eqref{1.0.1}, although they agree for solutions belonging to $H^s(\R)$ for $s$ sufficiently large. 

\begin{remark}
    The solutions of \eqref{3.3.3} and \eqref{1.0.1} also agree in other function space, such as Besov spaces, see \cite{dan}. One of the reasons of the present work considers only Sobolev space is because it can be seen as a dual work for finite dimensional manifolds of the results in \cite{const2000-1}.
\end{remark}

The observations made above are well known facts in the literature of the CH equation, but in view of their importance in the development of this manuscript, we want to give them the needed attention.

\begin{proposition}\label{prop3.1}
    Let $u\in C^0(H^4(\R),[0,T))\cap C^1(H^3(\R),[0,T))$. Then $u$ is a classical solution of the CH equation \eqref{1.0.1} if and only if $u$ is a classical solution of the non-local equation \eqref{3.3.3}. Moreover, in such a class, the Cauchy problem
    \bb\label{3.3.4}
\left\{
\ba{l}
m_t+2u_xm+um_x=0,\\
\\
u(x,0)=u_0(x)
\ea
\right.
\ee
is equivalent to
\bb\label{3.3.5}
\left\{
\ba{l}
\ds{u_t+uu_x+\p_x\Lambda^{-2}\Big(u^2+\f{u_x^2}{2}\Big)=0},\\
\\
u(x,0)=u_0(x).
\ea
\right.
\ee
\end{proposition}

In other words, proposition \ref{prop3.1} says that \eqref{1.0.1} and \eqref{3.3.3} are the same object in the class $C^0(H^4(\R),[0,T))\cap C^1(H^3(\R),[0,T))$.

The Cauchy problem \eqref{3.3.5} is more convenient to address the questions raised in the Introduction. In fact, in view of the tools developed by Kato \cite{kato}, we can establish the existence and uniqueness of a solution $u\in {\cal B}^s:=C^0(H^s(\R),[0,T))\cap C^{1}(H^{s-1}(\R),[0,T))$, $s>3/2$, for \eqref{3.3.5} emanating from an initial datum $u_0\in H^{s}(\R)$ \cite[Theorem 3.2]{blanco}. While any function in ${\cal B}^s$ is $C^1$ with respect to $t$, its regularity regarding $x$ is controlled by $s$. Therefore, taking $s$ sufficiently large we can reach to a higher regularity of the solution with respect to $x$, making it also a solution for \eqref{3.3.4}. See also \cite{freire-AML}.

It is time to drive back to PSS equations. As we have already pointed out, we must observe that several notions in this field were introduced, and have been used assuming, implicitly or explicitly, smooth solutions see \cite[Definition 2.4]{reyes2011}, \cite[page 89]{keti2015} \cite[page 89]{kamran}, and \cite[page 2]{tarcisio} and \cite[page 2]{kah}, respectively. On the other hand, our paper aims at seeing \eqref{3.3.3} as a PSS equation and thus, we need to look for notions that do not require $C^\infty$ regularity in the studied objects.

\begin{definition}{\tt($C^k$ PSS modelled by ${\cal B}$ and {\cal B}-PSS equation, \cite[Definition 2.1]{freire-ch})}\label{def3.4}
Let ${\cal B}$ be a function space, where their elements are $C^k$ functions. A differential equation \eqref{3.3.1}, for a dependent variable $u\in{\cal B}$, is said to describe a pseudospherical surface of class $C^k$ modelled by ${\cal B}$, $k\in\N$, or it is said to be of pseudospherical type modelled by ${\cal B}$, if it is a necessary and sufficient condition for the existence of functions $f_{ij}$, $1\leq i\leq 3,\,\,,1\leq j\leq 2$, depending on the solution $u$ of the equation and its derivatives, such that:
\begin{enumerate}\letra
    \item the functions $f_{ij}$ are $C^k$ with respect to their arguments;
    \item the forms 
    \bb\label{2.3.2}
\omega_i=f_{i1}dx+f_{i2}dt,\quad 1\leq i\leq 3,
\ee
satisfy the structure equations of a pseudospherical surface of Gaussian curvature ${\cal K}=-1$, that is,
\bb\label{2.3.3}
d\omega_1=\omega_3\wedge\omega_2,\quad d\omega_2=\omega_1\wedge\omega_3,\quad d\omega_3=\omega_1\wedge\omega_2;
\ee
    \item the condition $\omega_1\wedge\omega_2\not\equiv0$ is satisfied.
\end{enumerate}
\end{definition}

If the function space is clear from the context and no confusion is possible, we maintain the original terminology introduced in the works by Tenenblat and co-authors and simply say PSS equation in place of ${\cal B}-$PSS equation.

Whichever function space ${\cal B}$ is, the first condition asks it to be a subset of $C^k$, that is the space who utterly controls the regularity of the surface. 

\begin{remark}\label{rem3.1}It is possible to find books in differential geometry requiring $C^2$ metrics for a surface, which would force the one-forms being $C^2$ \cite[Theorem 4.24, page 153]{ku}. However, \cite[Theorems 10-19 and 10-19, page 232]{gug} and \cite[Theorem 10-18, page 232]{gug} require $C^1$ regularity of the one-forms defining a surface (and thus, a $C^1$ metric). It is worth noticing that this is the same regularity required by Hartman and Wintner \cite[page 760]{hartman}, who proved a sort of Bonnet theorem requiring $C^1$ metric of a surface defined on a domain in $\R^2$. 
\end{remark}

\begin{remark}\label{rem3.2}
    The second condition in definition \ref{def3.4} is satisfied if we are able to find functions $\mu_1$, $\mu_2$ and $\mu_3$, depending on $u$ and its derivatives up to a finite order, vanishing identically on the solutions of the equation, that is,
    $$
     d\omega_1-\omega_3\wedge\omega_2=\mu_1dx\wedge dt,\,\,d\omega_2-\omega_1\wedge\omega_3=\mu_2dx\wedge dt,\,\,d\omega_3-\omega_1\wedge\omega_2=\mu_3dx\wedge dt,
    $$
    and
    $$
    \mu_1\big|_{\eqref{3.3.1}}\equiv0,\,\,\,\mu_2\big|_{\eqref{3.3.1}}\equiv0\,\,\,\,\mu_3\big|_{\eqref{3.3.1}}\equiv0.
    $$
   
\end{remark}

\begin{remark}\label {rem3.3}
    In practical terms, the components of the functions $f_{ij}$, jointly with the conditions in Definition \ref{def3.4}, tells us the regularity we have to ask from the solution of the Cauchy problem in order to define a PSS. The final regularity that can be achieved is dictated by these coefficients and that required to grant the existence of solutions from the available tools for proving their well-posedness.
\end{remark}

\begin{remark}\label {rem3.4}
    The third condition is present for technical reasons, to avoid the situation $d\omega_3=0$, which would imply that $\omega_1=\al\omega_2$, for some $\al\in\R$. In practical aspects, this condition has to be verified case by case, depending on the solution. Despite being technical, this requirement truly ensures a surface structure in definition \ref{def3.4}.
\end{remark}

While definition \ref{def3.4} of ${\cal B}-$PSS equation has made only a minor modification in the previous one (that by Chern and Tenenblat), the same cannot be said about our proposed notion for a generic solution.

\begin{definition}{\tt(Generic solution, \cite[Definition 2.2]{freire-ch})}\label{def3.5}
A function $u:U\rightarrow\R$ is called {\it generic solution} for the ${\cal B}-$PSS equation ${\cal F}[v]=0$ if:
\begin{enumerate}\letra
\item $u\in{\cal B}$;
\item It is a solution of the equation. In other words, ${\cal F}[u]\equiv0$;
\item The one-forms \eqref{2.3.2} are $C^k$ on $U$;
\item There exists at least a simply connected open set $\Omega\subseteq U$ such that $\omega_1\wedge\omega_2(u(p))\neq0$, for each $p\in\Omega$.
\end{enumerate}

The condition $\omega_1\wedge\omega_2(u(p))\neq0$ has to be understood as follows: the one forms given in \eqref{2.3.2} usually depend on $(x,t,u,u_{(1)},\cdots,u_{(n)})$ for some $n$, where $u$ is a solution of the equation and $u_{(1)},\cdots,u_{(n)}$ denotes derivatives of $u$ up to order $n$. Let $p=(x,t)$ any point on the domain of the solution. By $\omega_1\wedge\omega_2(u(p))\neq0$ we mean
    $$\Big(f_{11}f_{22}-f_{12}f_{21}\Big)(p,u(p),u_{(1)}(p),\cdots,u_{(n)}(p))\neq0.$$

    Henceforth, any forms dealt with herein have to be understood in a similar sense. For a better discussion, see \cite[pages 77-78]{reyes2000}. 

A solution that is not generic is said to be {\it non-generic}.
\end{definition}

Let us show that the CH equation \eqref{1.0.1} is a $C^0(H^4(\R),[0,T))\cap C^{1}(H^{3}(\R),[0,T))-$PSS equation. 

\begin{example}\label{example3.4} 
We begin with the following observation: The minimum of regularity we can require to define a surface is $C^1$, see \cite[Theorems 10-19 and 10-19, page 232]{gug}. Therefore, the component functions of the one-forms \eqref{2.2.2} have to be of this order, which in particular, implies $m\in C^1$. As such, $u$ has to be at least $C^3$ with respect to $x$ and $C^1$ with respect to $t$, with continuous mixed derivatives. As a result, the CH equation is a PSS equation modelled by the function space  ${\cal B}:=C^{3,1}(U)$ and $u$ is a generic solution for the equation, bringing to $\Omega$ the structure of a PSS, in the following sense: $u$ is defined on $\Omega$ and  the pullback of the one-forms by $u$ and its derivatives evaluated on $\Omega$ satisfies the condition $\omega_1\wedge\omega_2\neq0$.

Let $\lambda\in\R\setminus\{0\}$ and consider the triad of one-forms \eqref{2.2.2}. A straightforward calculation shows that
\bb\label{3.3.7}
\ba{lcl}
d\omega_1-\omega_3\wedge\omega_2&=&\Big(m_t+2u_xm+um_x\Big)dx\wedge dt,\\
\\
d\omega_2-\omega_1\wedge\omega_3&=&0,\\
\\
d\omega_3-\omega_1\wedge\omega_2&=&-\Big(m_t+2u_xm+um_x\Big)dx\wedge dt,
\ea
\ee
and
\bb\label{3.3.8}
\omega_1\wedge\omega_2=-\Big(\f{\lambda}{2}+\f{1}{2\lambda}-m\Big)u_x dx\wedge dt.
\ee
Moreover, if $u$ is a solution of the CH equation, we conclude that  $\omega_1\wedge\omega_2=0$ if and only if 
$$m=\f{\lambda}{2}+\f{1}{2\lambda}\quad\text{or}\quad u_x=0,$$
that, substituted into \eqref{2.0.1}, implies
\bb\label{3.3.9}
u(x,t)=c,
\ee
for some constant $c$. According to \cite[Theorem 2.2]{freire-ch}, if $u\in C^0(H^4(\R),[0,T))\cap C^{1}(H^{3}(\R),[0,T))$ is a non-trivial solution of the CH equation, then not only $m$ cannot be constant on some simply connected, open set $\Omega\subseteq T\times[0,T)$, but also either $u_x\big|_\Omega>0$ or $u_x\big|_\Omega<0$. As a result, $u$ is a generic solution in the sense of Definition \ref{def3.5}.

\end{example}

Example \ref{example3.4} does not necessarily show that \eqref{3.3.3} can be seen as a PSS equation. However, if we restrict the solutions of the CH equation \eqref{1.0.1} to the class ${\cal B}=C^0(H^4(\R),[0,T))\cap C^1(H^3(\R),[0,T))\subseteq C^{3,1}(\R\times[0,T))$ as in proposition \ref{prop3.1}, then the {\it same} one-forms \eqref{2.2.2} give 
\bb\label{3.3.10}
\ba{lcl}
d\omega_1-\omega_3\wedge\omega_2&=&\ds{(1-\p_x^2)\Big(u_t+uu_x+\p_x\Lambda^{-2}\Big(u^2+\f{u_x^2}{2}\Big)\Big)dx\wedge dt,}\\
\\
d\omega_2-\omega_1\wedge\omega_3&=&0,\\
\\
d\omega_3-\omega_1\wedge\omega_2&=&\ds{-(1-\p_x^2)\Big(u_t+uu_x+\p_x\Lambda^{-2}\Big(u^2+\f{u_x^2}{2}\Big)\Big)dx\wedge dt,}
\ea
\ee
and thus \eqref{3.3.3} is a PSS equation in the sense of definition \ref{def3.4}.

In fact, we have the following result.

\begin{theorem}\label{teo3.1}
    Let $T>0$ and consider the function space  ${\cal B}=C^0(H^4(\R),[0,T))\cap C^1(H^3(\R),[0,T))\subseteq C^{3,1}(\R\times[0,T))$. Then the CH equation \eqref{1.0.1} is a PSS equation modelled by ${\cal B}$ if and only if the non-local evolution equation \eqref{3.3.3} is a PSS equation modelled by ${\cal B}$. Moreover, they describe exactly the same PSS, in the sense that $u\in{\cal B}$ is a generic solution of \eqref{1.0.1} if and only if it is a generic solution of \eqref{3.3.3}.
\end{theorem}

While theorem \ref{teo3.1} tells us that the geometric object described by \eqref{3.3.7} is identical to that given by \eqref{3.3.10}, it does not say when or how we can determine whether we really have a PSS from a solution. Moreover, finding a solution of a highly non-linear equation like \eqref{1.0.1} is a rather non-trivial task. 

One of the advantages of the modern methods for studying evolution PDEs is the fact that we can extract much information about properties of solutions, that we do not necessarily know explicitly, from the knowledge of an initial datum. The equivalence between Cauchy problems given by proposition \ref{prop3.1} and theorem \ref{teo3.1} suggest that we could have qualitative information from the surface provided that we know an initial datum.

\begin{theorem}\label{teo3.2}
Let $u_0\in H^4(\R)$ be a non-trivial initial datum, and consider the Cauchy problem \eqref{3.3.4}. Then there exists a value $T>0$, uniquely determined by $u_0$, and an open strip of height $T$ ${\cal S}=\R\times(0,T)$, such that the forms \eqref{2.2.2} are uniquely determined by $u_0$, defined on ${\cal S}$, and of class $C^1$. Moreover, the Hamiltonian 
${\cal H}_1$, given in \eqref{2.0.2}, provides a conserved quantity on the solutions of problem \eqref{3.3.4}. 
\end{theorem}

By a non-trivial function we mean one that is not identically zero.

The geometric meaning of theorem \ref{teo3.2} is the following: given a regular curve 
\bb\label{3.3.11}
\gamma(x)=(x,0,u_0(x)),\quad u_0\in H^4(\R),
\ee
let $\Gamma:=\{\gamma(x),\,x\in\R\}$. Then we can uniquely determine a solution $u(x,t)$ of the CH equation such that $\Gamma\subseteq\overline{\Gr(u))}$, where 
$$\Gr(u)=\{(x,t,u(x,t)),\,x\in\R,\,t>0\}$$
and $\overline{\Gr(u)}$ denotes the closure of $\Gr(u)$. 

Even though the existence of the forms \eqref{2.2.2} over a domain ${\cal S}\neq\emptyset$ is a necessary condition for endowing\footnote{By endowing ${\cal S}$ with a PSS structure we mean that the restriction of $u$ to ${\cal S}$ is such that the one forms satisfies the conditions for defining a PSS.} ${\cal S}$ with the structure of a PSS, it is not sufficient, since the condition $\omega_1\wedge\omega_2\neq0$ is fundamental for such, and theorem \ref{teo3.2} says nothing about it. 

It is worth mentioning that a solution $u$ of the CH equation subject to an initial datum in $H^4(\R)$ is unique and its domain is determined by the initial datum \cite[Proposition 2.7]{const1998-1} and it has to be considered intrinsically with its domain. Moreover, the invariance of the conserved quantity ${\cal H}_1$ in \eqref{2.0.2} implies $u_x(\cdot,t)\in L^2(\R)$, for each $t$ for which the solution exists. Let us fix $t_0\in(0,T)$. Then $u_x(x,t_0)\rightarrow0$ as $|x|\rightarrow\infty$. Since ${\cal H}_1(0)>0$, then $u(\cdot,t_0)\not\equiv0$ and cannot be constant. Therefore, $u_x(\cdot,t_0)$ cannot be constant either. As a result, we conclude the existence of two points $x_0$ and $x_1$ such that the mean value theorem implies $u_x(x_0,t_0)=0$, whereas for the other we have $u_x(x_1,t_0)\neq0$, say $u_x(x_1,t_0)>0$. The continuity of $u_x$ then implies the existence of an open and simply connected set $\Omega$ such that $u_x(\cdot,\cdot)\big|_{\Omega}>0$ is not constant.

These comments prove the following result.
\begin{corollary}\label{cor3.1}
Assume that $u_0$ is a solution satisfying the conditions in theorem \ref{teo3.2} and let $u$ be the unique solution of \eqref{3.3.5}. Then $u_x(\cdot,\cdot)$ vanishes at a non-countable number of points of ${\cal S}$. Moreover, there exist open and simply connected subsets $\Omega\subseteq U$ such that $u_x(x,t)$ does not vanish for any $(x,t)\in\Omega$.
\end{corollary}

We have an even stronger result coming from the precedent lines.
\begin{corollary}\label{cor3.2}
Any solution of \eqref{3.3.5}, emanating from a non-trivial initial datum $u_0\in H^4(\R)$, is a generic solution in the sense of definition \ref{def3.4}.
\end{corollary}

Theorem \ref{teo3.2} and its corollaries show that any non-trivial initial datum determines a PSS, compare with \cite[Theorem 2.2]{freire-ch}, and their proof is given in subsection \ref{subsec5.2}. Due to \cite[Theorem 2.2]{freire-ch}, these results are somewhat expected. The same, however, cannot be said about our next proclamation.

\begin{theorem}\label{teo3.3}
    Assume that $u_0\in H^4(\R)$ is a non-trivial, compactly supported initial datum, with $[a,b]=\supp(u_0)$ and $u$ be the corresponding solution of \eqref{3.3.4}. Then there exists two $C^1$ curves $\gamma_+,\gamma_-:[0,T)\rightarrow \overline{{\cal S}}$, and two $C^1$ functions $E_+,\,E_-:[0,T)\rightarrow\R$, where $T\in\R$ and ${\cal S}\subseteq\R^2$ are given in Theorem \ref{teo3.2}, such that:
\begin{enumerate}\letra
    \item $\pi_1(\gamma_-(t))<\pi_1(\gamma_+(t))$, for any $t\in[0,T)$, where $\pi_1:\R^2\rightarrow\R$ is the canonical projection $\pi_1(x,t)=x$;
    \item $\gamma_\pm'(t)\neq 0$, for any $t\in(0,T)$;
    \item On the left of $\gamma_-$, the first fundamental form is given by
    \bb\label{3.3.12}
    \ba{lcl}
    g&=&\ds{\f{1}{4}\Big(\lambda+\f{1}{\lambda}\Big)dx^2+2\Big(\f{\lambda}{2}+\f{1}{2\lambda}\Big)\Big[\Big(\f{\lambda}{2}-\f{1}{2\lambda}\Big)E_-(t) e^{x}-\f{1}{2}-\f{\lambda^2}{2}\Big]dxdt}\\
    \\
    &&\ds{+\Big[E_-(t)^2 e^{2x}+\Big(\Big(\f{\lambda}{2}-\f{1}{2\lambda}\Big)E_-(t)e^{x}-\f{1}{2}-\f{\lambda^2}{2}\Big)^2\Big]dt},
    \ea
    \ee
    
    \item On the right of $\gamma_+$, the first fundamental form is given by
    \bb\label{3.3.13}
    \ba{lcl}
    g&=&\ds{\f{1}{4}\Big(\lambda+\f{1}{\lambda}\Big)dx^2+2\Big(\f{\lambda}{2}+\f{1}{2\lambda}\Big)\Big[\Big(\f{\lambda}{2}-\f{1}{2\lambda}\Big)E_+(t) e^{-x}-\f{1}{2}-\f{\lambda^2}{2}\Big]dxdt}\\
    \\
    &&\ds{+\Big[E_+(t)^2 e^{-2x}+\Big(\Big(\f{\lambda}{2}-\f{1}{2\lambda}\Big)E_+(t)e^{-x}-\f{1}{2}-\f{\lambda^2}{2}\Big)^2\Big]dt}.
    \ea
    \ee
\end{enumerate}  
\end{theorem}

If we denote by $(g)$ the matrix of the first fundamental form and fix $t\in(0,T)$, then the metrics \eqref{3.3.12} and \eqref{3.3.13} can be written in a unified way, that is,
$$
(g)=\begin{pmatrix}
\ds{\f{1}{4}\Big(\lambda+\f{1}{\lambda}\Big)}&\ds{ -\f{1}{4}(1+\lambda^2)\Big(\lambda+\f{1}{\lambda}\Big)}\\
\\
\ds{-\f{1}{4}(1+\lambda^2)\Big(\lambda+\f{1}{\lambda}\Big)}& \ds{\f{1}{4}\Big(1+\f{1}{2\lambda}\Big)}
\end{pmatrix}+O(e^{-|x|})=:(g_0)+O(e^{-|x|}),
$$
as $|x|\rightarrow\infty$, meaning that the matrix $(g)$ is an $O(e^{-|x|})$ perturbation of the singular matrix $(g_0)$ as $|x|\rightarrow+\infty$. Therefore, the metric determined by a compactly supported initial datum becomes asymptotically singular, for each fixed $t\in(0,T)$. Hence, for $|x|\gg 1$ and $t$ fixed, the components of the metric behave like the famous peakon solutions of the CH equation.

\begin{theorem}\label{teo3.4}
If $u_0\in H^4(\R)$ and for some $x_0\in\R$, we have
\bb\label{3.3.14}
u_0'(x_0)<-\f{\|u_0\|_{1}}{\sqrt{2}},
\ee
then there exists $0<T_m<\infty$ such that the metric \eqref{2.2.3}, determined by the solution o \eqref{3.3.4}, blows up as $t\rightarrow T_m$. More precisely, the coefficients $g_{11}$ and $g_{12}$ are uniformly bounded whereas
\bb\label{3.3.15}
\liminf_{t\rightarrow T_m}\Big(\sup_{x\in\R}g_{22}(x,\tau)\Big)=+\infty.
\ee
\end{theorem}

Expression \eqref{3.3.15} says that the metric blows up for a finite value of $t$ and then, the surface can only be defined on a proper subset of $\R^2$. 

While Theorem \ref{teo3.3} tells us that the metric determined by an initial datum becomes asymptotically singular for each fixed $t$ as long as the solution exists, theorem \ref{teo3.4} shows us a different sort of singularity, in which the metric blows up over a strip of finite height. Our next result, however, informs us that a compactly supported initial datum actually leads to a singularity of the metric similar to that established in Theorem \ref{teo3.4}.

\begin{theorem}\label{teo3.5}
    If $u_0\in H^4(\R)$ is a non-trivial, compactly supported initial datum, then the metric \eqref{2.2.3}, determined by the solution o \eqref{3.3.4}, blows up within a strip of finite height.
\end{theorem}

Theorems \ref{teo3.4} and \ref{teo3.5} tell us the existence of a height for which the co-frame of dual forms $\omega_1$ and $\omega_2$ are well defined, but their corresponding metric becomes unbounded near some finite height, meaning that the metric, and the forms as well, are only well defined on a certain strip with infinite length, but finite height.

A completely different scenario is given by our next result.

\begin{theorem}\label{teo3.6}
    Let $m_0\in H^{2}(\R)\cap L^1(\R)$ and $u$ be the corresponding solution of \eqref{3.3.4}. If $m_0(x)\geq0$ or $m_0(x)\leq0$, then \eqref{2.2.2} are $C^1$ one-forms defined on ${\cal S}=\R\times(0,\infty)$. Moreover, for any $R>0$, there exists a simply connected set ${\cal R}\subseteq\R^2$ such that $\sqrt{x^ 2+t^ 2}>R$, for any $(x,t)\in{\cal R}$, and $u_x\big|_{\cal R}>0$ or $u_x\big|_{\cal R}<0$. 
\end{theorem}

Theorem \ref{teo3.6} says that subsets of the domain of the solution of the CH equation that can be endowed with a PSS structure cannot be contained in any compact set. In view of this result, regions arbitrarily far away from the origin may be endowed with the structure of a PSS.

\section{Preliminaries}\label{sec4}

In this section we present auxiliary results that will help us to prove technical theorems and will be of vital importance in order to establish our main results.

\begin{lemma}{\tt(\cite[Proposition 2.7]{const1998-1})}\label{lema4.1}
If $u_0\in H^4(\R)$, then there exists a maximal time $T=T(u_0)>0$ and a unique solution $u$ to the Cauchy problem \eqref{3.3.5} such that $u=u(\cdot,u_0)\in C^{0}(H^{4}(\R),[0,T))\cap C^{1}(H^{3}(\R),[0,T))$. Moreover, the solution depends continuously on the initial data, in the sense that the mapping $u_0\mapsto u(\cdot,u_0):H^{4}(\R)\rightarrow C^{0}(H^{4}(\R),[0,T))\cap C^{1}(H^{3}(\R),[0,T))$ is continuous.
\end{lemma}

\begin{remark}\label{rem4.1}
We observe that if, instead of $u_0\in H^4(\R)$, we assume $u_0\in H^s(\R)$, $s>3/2$, we would then conclude that $u\in C^{0}(H^{s}(\R),[0,T))\cap C^{1}(H^{s-1}(\R),[0,T))$, for the same $T$, see \cite[Theorem 3.2]{blanco}. 
\end{remark}

\begin{lemma}{\tt(\cite[Theorem 1.1]{freire-AML})}\label{lema4.2}
Assume that $m_0\in H^2(\R)\cap L^1(\R)$. If $m_0(x)\geq0$ or $m_0(x)\leq0$, for any $x\in\R$, then the corresponding solution $u$ of the CH equation exists globally. In other words, the solution $u$ of the CH equation belongs to the class $C^{0}(H^{4}(\R),[0,\infty))\cap C^{1}(H^{3}(\R),[0,\infty))$. 
\end{lemma}

\begin{lemma}\label{lema4.3}{\tt(\cite[Theorem 3.1]{const2000-1})} Let $u_0\in H^3(\R)$ and $[0,T)$ be the maximal interval of existence of the corresponding solution of \eqref{3.3.5}. Then
\bb\label{4.2.2}
\left\{
\ba{lcl}
q_t(x,t)&=&u(q,t),\\
\\
q(x,0)&=&x,
\ea
\right.
\ee
has a unique solution $q\in C^1(\R\times[0,T),\R)$. Moreover, for every fixed $t\in[0,T)$, the function $q(\cdot,t)$ is an increasing diffeomorphism of the line.
\end{lemma}

\begin{lemma}\label{lema4.4}{\tt(\cite[Theorem 4.2]{const1998-2})} 
Given an initial datum $u_0\in H^3(\R)$ satisfying \eqref{3.3.15}, then the corresponding solution $u$ of the CH equation subject to $u(x,0)=u_0(x)$ breaks at finite time, that is, there exists a finite time $T_m>0$ such that 
\bb\label{4.2.1}
\lim_{t\rightarrow T_m}\inf\Big(\inf_{x\in\R} u_x(t,x)\Big)=-\infty.
\ee
\end{lemma}

\begin{lemma}{\tt(\cite[Theorem 2.1]{const1998-2})}\label{lema4.5}
Let $T>0$ and $v\in C^1(H^2(\R),[0,T))$ be a given function. Then, for any $t\in[0,T)$, there exists at least one point $\xi(t)\in \R$ such that
\begin{align}\label{4.2.2}
    y(t) :=\inf\limits_{x\in\R}v_x(x,t) = v_x(\xi(t),t)
\end{align}
and the function $y$ is almost everywhere differentiable in $(0,T)$, with $y'(t)=v_{tx}(\xi(t),t)$ almost everywhere in $(0,T)$.
\end{lemma}

\begin{lemma}{\tt(\cite[Theorem 1.4]{him-cmp})}\label{lema4.6}
    If $u_0\in H^4(\R)$, is compactly supported, then there exist $C^1$ real valued functions $E_\pm$ such that
    $$
    u(x,t)=\left\{\ba{lcl}
    E_+(t)e^{-x},&\text{for} & x>q(b,t),\\
    \\
    E_-(t)e^{x},&\text{for} & x<q(a,t),
    \ea\right.
    $$
where $q(\cdot,\cdot)$ is the function given in Lemma \ref{lema4.3}, for any $t>0$ such that the solution exists.
\end{lemma}

The original statement of Lemma \ref{lema4.6} says that $s>5/2$ and the functions $E_\pm$ are continuous. It is immediate then its validity for $s=4$, that is our case, and a careful analysis on the proof of \cite[Theorem 1.4]{him-cmp} reveals that the functions are continuously differentiable.

\section{Proof of the main results}\label{sec5}

\subsection{Proof of theorem \ref{teo3.1}}\label{subsec5.1}

From \eqref{3.3.2}, $u\in{\cal B}$ is a solution of \eqref{1.0.1} in the sense of definition \ref{def3.4} if and only if it is a solution of \eqref{3.3.3} in the same sense. Let $w_0\in H^4(\R)$, $u_1$ and $u_2$ be the corresponding solutions of \eqref{1.0.1} and \eqref{3.3.3}, respectively, subject to the same initial condition $u_1(x,0)=u_2(x,0)=w_0(x)$. Proposition \ref{prop3.1} combined with lemma \ref{lema4.1} inform us that $u_1=u_2$ and this is the only solution for both equations satisfying the given initial condition. As a result, they determine the same forms $\omega_1,\omega_2,\omega_3$, and the same PSS as well.

\subsection{Proof of theorem \ref{teo3.2}}\label{subsec5.2}

Lemma \ref{lema4.1}, jointly with remark \ref{rem4.1} and Theorem \ref{teo3.1}, assures that \eqref{3.3.5} has a unique solution $u\in C^0(H^4(\R),[0,T))\cap C^1(H^3(\R),[0,T))\subseteq C^{3,1}(\R\times[0,T))$, for a $T$ uniquely determined by $u_0$. We then conclude that the one-forms \eqref{2.2.2} are $C^1$ and defined on the open and connected set ${\cal S}=\R\times(0,T)$.

Due to $u_0\in H^4(\R)$, then $\|u_0\|_1<\infty$. Moreover, the functional ${\cal H}_1(t)$, given in \eqref{2.0.2}, is constant, that is, ${\cal H}_1(t)={\cal H}_1(0)$, $t\in(0,T)$. Given that $t\mapsto{\cal H}(t)=\|u\|_1^2/2$ is invariant, we conclude $\|u\|_1=\|u_0\|_1$.

\subsection{Proof of Theorem \ref{teo3.3}}\label{subsec5.3}

Let $u$ be the corresponding solution of the CH equation subject to $u(x,0)=u_0(x)$ and $q$ be the function given by Lemma \ref{lema4.3}.

Define $\varphi(x,t):\R\times[0,T)\rightarrow \R\times[0,T)$ by $\varphi(x,t)=(q(x,t),t)$. Then $\varphi$ is a bijection fixing $\R\times\{0\}$ and $\varphi\big|_{\R\times(0,T)}$ is a $C^1$ diffeomorphism, see \cite[Theorem 3.1]{freire-ch}.

Let $\gamma_\pm:[0,T)\rightarrow\overline{{\cal S}}$ be given by $\gamma_-(t)=\varphi(a,t)$ and $\gamma_+(t)=\varphi(b,t)$. Then $\gamma_-'(t)=(u(\varphi(a,t)),1)$ and $\gamma_+'(t)=(u(\varphi(b,t)),1)$. Again, by Lemma \ref{lema4.3} we have
$$
\pi_1(\gamma_-(t))=q(a,t)<q(b,t)=\pi_1(\gamma_+(t)),
$$
for each $t\in(0,T)$.

Let $p\in {\cal S}$ be a point on the left of $\gamma_-$. This then implies that
$$x:=\pi_1(p)<\pi_1(\gamma_-(t))=q(a,t).$$

By Lemma \ref{lema4.6} we have $u(x,t)=E_-(t)e^{x}$, that substituted into \eqref{2.2.3} gives \eqref{3.3.12}. To get \eqref{3.3.13} we proceed mimetically as before and for this reason is omitted.

\subsection{Proof of theorem \ref{teo3.4}}
Let us define 
\bb\label{5.4.1}
y(t)=\inf_{x\in\R} u_x(x,t).
\ee

By lemma \ref{lema4.5} we can find $\xi(t)$ (despite the notation, it is not a function, see \cite[Theorem 2.1]{const1998-2}) such that $y(t)=u_x(\xi(t),t)$ and it is an a.e. $C^1$ function. Moreover, \cite[Theorem 4.2]{const1998-2} shows in its demonstration that $y$ is Lipschitz and $y(0)\leq u_0'(x_0)<0$.

Differentiating \eqref{3.3.3} with respect to $x$ and using $y(t)$ above, we obtain
$$
y'(t)+\f{y(t)^2}{2}=u(\xi(t),t)^2-\Big(\p_x\Lambda^{-2}\Big(u^2+\f{u_x^2}{2}\Big)\Big)(\xi(t),t).
$$
In \cite[page 240]{const1998-1} it was proved that $y(t)$ satisfies the differential inequality
$$
y'(t)\leq -\f{\epsilon}{4}y(t)^2,
$$
for some $\epsilon\in(0,1)$, implying that it is a negative and non-increasing function satisfying the inequality
\bb\label{5.4.2}
\f{\epsilon}{4}t+\f{1}{y(0)}\leq \f{1}{y(t)}.
\ee

Since $y(t)<y(0)<0$, then \eqref{5.4.2} is only valid for a finite range of values for $t$. As a result, we conclude the existence of $T_m$ such that \eqref{5.4.2} holds for $t\in(0,T_m)$, and then, the solution $u$, as a function of $t$, is only defined on $(0,T_m)$.

On the other hand, \eqref{5.4.2} can be seen in a slightly different way, since it implies
$$
0\leq\f{\epsilon}{4}t-\f{1}{y(t)}\leq -\f{1}{y(0)},
$$
which tells us that $y(t)\rightarrow-\infty$ before $t$ reaches $-4/(\epsilon y(0))$ (which gives an upper bound to $T_m$). As a result, if $(t_k)_{k}\subseteq(0,T_m)$ is a convergent sequence to $T_m$, we then have $y(t_k)\rightarrow-\infty$ as $k\rightarrow\infty$. This, in particular, is nothing but \eqref{4.2.1}.

Let us evaluate the coefficients $g_{ij}$ of the metric \eqref{2.2.3} at $x=\xi(t)$. The Sobolev Embedding Theorem (see lemma \ref{lema3.1}) implies that $u$ is uniformly bounded in $(0,T_m)$ by $\|u_0\|_1$. Since $x=\xi(t)$ is a point of minima of the function $u_x(\cdot,t)$, we conclude that $u_{xx}(\xi(t),t)=0$ and thus, $m(\xi(t),t)=u(\xi(t),t)$ is bounded as well. As a result, we conclude that both $g_{11}(\xi(t),t)$ and $g_{12}(\xi(t),t)$ are uniformly bounded for $t\in(0,T_m)$.

A different situation occurs with $g_{22}$. The previous arguments show that $g_{22}(\xi(t),t)=u_x(\xi(t),t)^2+B(u(\xi(t),t))$, where $B(u(\xi(t),t))$ encloses the uniformly bounded remaining terms of the metric in $(0,T_m)$.

For any sequence $(t_k)_{k}\subseteq(0,T_m)$ convergent to $T_m$, we have
$$
\sup_{x\in\R} g(x,t_k)\geq g_{22}(\xi(t_k),t_k)=u_x(\xi(t_k),t_k)^2+B(u(\xi(t_k),t_k))\rightarrow +\infty
$$
as $k\rightarrow\infty$, showing that
$$
\sup_{(x,t)\in\R\times[0,T_m)}g_{22}(x,t)=\lim_{t\rightarrow T_m}\inf_{\tau\geq t}\Big(\sup_{x\in\R}g_{22}(x,\tau)\Big)=+\infty.
$$

\subsection{Proof of theorem \ref{teo3.5}}\label{subsec5.4}

From \eqref{2.2.2} we have $f_{32}(x,t)=-u_{x}(x,t)$, and, as a result,
\bb\label{5.5.1}
\|f_{32}(\cdot,t)\|_\infty=\|u_x(\cdot,t)\|_\infty.
\ee

Therefrom, for each $t$ such that the solution exist, we have
\bb\label{5.5.2}
\int_0^t \|f_{32}(\cdot,\tau)\|_\infty\,d\tau=\int_0^t\|u_x(\cdot,\tau)\|_\infty\,d\tau.
\ee

By Theorem \ref{teo3.1} and the conditions on the initial datum, we conclude that the function defined in \eqref{5.5.2} is continuous. Let us prove the existence of a height $T_m<\infty$ such that $\|f_{32}(\cdot,t)\|_\infty\rightarrow\infty$ as $t\rightarrow T_m$.

The maximal height $T_m$ corresponds to the maximal time of existence of the solution. Following \cite[Corollary 1.1]{him-cmp} or \cite[Theorem 6.1]{bran-imrn}, the conditions on the initial datum in Theorem \ref{teo3.5} imply that the solution $u$ can only exist for a finite time $T_m$, implying on the existence of a maximal height $T_m$ for the strip in Theorem \ref{teo3.2}.

By \cite[Corollary 1.1, Eq. (1.20)]{him-cmp} we then have
$$
\int_0^{T_m}\|u_x(\cdot,\tau)\|_\infty\,d\tau=\infty.
$$

On the other hand, the singularities of the solution arise only in the form of wave breaking. Moreover, we have the equivalence (e.g, see \cite[page 525, Eq. (3.7)]{mol})
\bb\label{5.5.3}
\int_0^{T_m}\|u_x(\cdot,\tau)\|_\infty\,d\tau=\infty \Longleftrightarrow \int_0^{T_m}\|y(\tau)\|_\infty\,d\tau=\infty,
\ee
where $y(\cdot)$ is given by \eqref{5.4.1}. Let $(t_k)_{k\in\N}\subseteq(0,T_m)$ be any sequence convergent to $T_m$. By \eqref{5.5.3}, \eqref{5.5.2}, \eqref{5.4.1} and Lemma \ref{lema4.5}, we have
$y(t_k)=u(\xi(t_k),t_k)<\infty$ and 
$$
\int_0^{t_k} \|f_{32}(\cdot,\tau)\|_\infty\,d\tau<\infty,
$$
for any $k\in\N$, but
$$
\lim_{k\rightarrow\infty}\int_0^{t_k} \|f_{32}(\cdot,\tau)\|_\infty\,d\tau=\infty,
$$
meaning that $|f_{32}(x,t)|$ becomes unbounded near some point of the horizontal line $\R\times\{T_m\}$. Since $g_{22}(x,t)\geq f_{32}(x,t)^2$, we have $\sup_{x\in\R}g(x,t_k)\geq f_{32}(\xi(t_k),t_k)^2\rightarrow\infty$ 
as $k\rightarrow\infty$, and we then get again
\bb\label{5.5.4}
\sup_{(x,t)\in\R\times[0,T_m)}g_{22}(x,t)=\lim_{t\rightarrow T_m}\inf_{\tau\geq t}\Big(\sup_{x\in\R}g_{22}(x,\tau)\Big)=+\infty,
\ee
which proves the result.

We can give a slightly different proof starting from \eqref{5.5.3}. In fact, that condition implies on the wave breaking of the solution. According to McKean \cite{mk1,mk2}, this only happens if and only if the points for which $m_0(x)$ is positive lies to the left of those that $m_0(x)$ is negative, see also \cite[Theorem 1.1]{jiang}. In other words, for some $x_0\in\R$, we have $m_0(x_0)\geq0$, for $x\leq x_0$, whereas for $x\geq x_0$ we have $m_0(x_0)\leq0$. By \cite[Theorem 3.3]{freire-ch}, we get back to \eqref{5.5.4}.

\subsection{Proof of theorem \ref{teo3.6}}

By lemma \ref{lema4.2}, $u$ is a global solution in the class $C^0(H^{4}(\R),[0,\infty))\cap C^0(H^{3}(\R),[0,\infty))$. In particular, it is defined on ${\cal S}=\R\times(0,\infty)$ and, therefore, the coefficients $f_{ij}$, $1\leq i\leq 3$, $1\leq j\leq 2$, of the one-forms \eqref{2.2.2} belong to the class $C^{3,1}(\R\times(0,\infty))\subseteq C^{1}(\R\times(0,\infty))$, and then, $g_{kl}\in C^{1}(\R\times(0,\infty))$, $1\leq k,l\leq 2$.

By corollary \ref{cor3.1} we know that $\{\omega_1,\omega_2\}$ cannot be linearly independent everywhere. Let $R>0$, $\overline{B}_R(0);=\{(x,t)\in U;\,\,x^2+t^2\leq R^2\}$, and $W_R:=U\setminus \overline{B}_R(0)$.

Suppose that for some $R>0$ we had $u_x\big|_{W_R}=0$. Then $u\big|_{W_R}=c$, for some $c\in\R$, and since $u\in L^2(\R)$, we would conclude that $c=0$, resulting in $u\big|_{{\cal R}}=0$, for any open set ${\cal R}\subseteq W_R$. Therefore, we can find numbers $t_0>R$ and $b>a>R$ such that $[a,b]\times\{t_0\}\subseteq{\cal R}$, $u(x,t_0)=u_t(x,t_0)=0$, $a\leq x\leq b$. From \eqref{3.3.3} we obtain
$$
\p_x\Lambda^{-2}\Big(u^2+\f{u_x^2}{2}\Big)(x,t)=-\Big(u_t+uu_x\Big)(x,t).
$$

Evaluating at $t=t_0$ and letting $x\in(a,b)$, we conclude that
$$
F(x):=\p_x\Lambda^{-2}\Big(u^2+\f{u_x^2}{2}\Big)(x,t_0)=-\Big(u_t+uu_x\Big)(x,t_0)\equiv0,
$$
implying $F'(x)=0$, $x\in(a,b)$. Since $\p_x^2\Lambda^2=\Lambda^{-2}-1$, we get
$$0=F'(x)=\Lambda^{-2}\Big(u^2+\f{u_x^2}{2}\Big)(x,t_0)=\f{1}{2}\int_\R\f{e^{-|x-y|}}{2}\Big(u^2+\f{u_x^2}{2}\Big)(y,t_0)dy,\quad x\in(a,b),$$
wherefrom we arrive at the conclusion $u(x,t_0)\equiv0$, $x\in\R$. This would then imply $\|u\|_1=0$ at $t=t_0$. The invariance of $\|u\|_1$ implies $u\equiv0$, that conflicts with $u_0$ being a non-trivial initial datum.

The contradiction above forces us to conclude that, for any $R>0$, we can find $(x_R,t_R)\in W_R$ such that $u_x(x_R,t_R)\neq0$, meaning that we either have $u_x(x_R,t_R)>0$ or $(x_R,t_R)<0$. Since $u_x$ is continuous, we can find a neighbourhood $V_R$ of $(x_R,t_R)$ such that $u_x\big|_{V_R}$ has the same sign. 

We now observe that $m$ cannot be constant, see Example \ref{example3.4}. As a result, for some open set ${\cal R}\subseteq V_R$ we can have $m\neq \f{\lambda}{2}+\f{1}{2\lambda}$. Then the pullback of $\omega_1$ and $\omega_2$ with respect to $u$ and its derivatives on ${\cal R}$ satisfies the condition $\omega_1\wedge\omega_2\neq0$.

\section{Finite height vs finite time of existence}\label{sec6}

The results proved in \cite{freire-ch} and those in theorems \ref{teo3.4} and \ref{teo3.5} suggest that the metric blows up as long as the solution develops a wave breaking. This is, indeed, the case.

\begin{theorem}\label{teo6.1}
    Let $u\in C^0(H^4(\R),[0,T))\cap C^1(H^3(\R),[0,T))$ be a solution of the CH equation and $g_{22}$ be the corresponding component of the metric tensor given in \eqref{2.2.3}. Then $g_{22}$ blows up within a strip of finite height if and only if $u$ breaks in finite time.
\end{theorem}

\begin{proof}
Let $q$ be the function given in Lemma \ref{lema4.3} and $\varphi(x,t)=(q(x,t),t)$ be the bijection given in the proof of Theorem \ref{teo3.3} (see subsection \ref{subsec5.3}). As long as the solution exists for $t>0$ and taking \eqref{2.0.1} into account, we have
$$
\f{d}{dt}m(\varphi(x,t))=(m_t+um_x)(\varphi(x,t))=-2(u_xm)(\varphi(x,t)),
$$
that is,
\bb\label{6.0.1}
m(\varphi(x,t))=m_0(x)e^{\ds{-2\int_0^t u_x(\varphi(x,\tau))d\tau}}.
\ee

Since $u\in C^0(H^4(\R),[0,T))\cap C^1(H^3(\R),[0,T))$ and $m_0=u_0-u_0''$, Lemma \ref{lema3.1} implies $\|m_0\|_{L^\infty}<\infty$. Moreover, we have $\|u(\cdot,t)\|_{L^\infty}<\|u(\cdot,t)\|_{H^1}=\|u_0(\cdot)\|_{H^1}.$
As a result, from \eqref{2.2.3} we have continuous functions $\kappa(u,\lambda)$ and $B(u,\lambda)$ such that
$$
u_x(x,t)^2\leq g(x,t)\leq u_x(x,t)^2+\kappa(u,\lambda)(m^2+m)+B(u,\lambda).$$

Let $\kappa_1=\sup\limits_{|s|\leq \|u_0(\cdot)\|_{H^1}} \kappa(s,\lambda)$ and $\kappa_2=\sup\limits_{|s|\leq \|u_0(\cdot)\|_{H^1}}B(s,\lambda)$. From the inequality above we arrive at

\bb\label{6.0.3}
u_x(x,t)^2\leq g(x,t)\leq
u_x(x,t)^2+\kappa_1(m^2+m)+\kappa_2,
\ee
that combined with \eqref{5.5.3} and \eqref{6.0.1} show that $g_{22}$ blows up in a strip of finite height if and only if $u_x$ blows up in finite time. Hence, we have 
$$
\sup_{(x,t)\in\R\times[0,T)}g_{22}(x,t)=\infty \Longleftrightarrow \liminf_{t\rightarrow T}\big(\inf_{x\in\R} u_x(x,t)\big)=-\infty.
$$

In particular, the maximal height of the strip coincides with the maximal time of existence of the solutions. 
\end{proof}

\section{Examples}\label{sec7}

We give two examples illustrating qualitative aspects of the surfaces determined by solutions of the CH equation once an initial datum is known.

\begin{example}\label{example7.1}
Let us consider $m_0(x)=e^{-x^{2}}$. As a consequence of \eqref{5.4.2}, $m(x,t)>0$ and so does the corresponding solution $u$. As a result of theorem \ref{teo3.1} and its corollaries, $u$ is a generic solution of the CH equation in the sense of definition \ref{def3.4}.

By theorem \ref{teo3.6}, the one-forms \eqref{2.2.2} are defined on ${\cal S}=\R\times(0,\infty)$ and they endow an infinite number of simply connected open sets $\Omega\subseteq U$ with the structure of a PSS.

\end{example}

\begin{example}\label{example7.2}
Let us now consider the family of functions $\phi_n(x)=e^{-nx^2}$, $n\in\N$ and $x\in\R$. As pointed out in \cite[Example 4.3]{const1998-2}, for $n$ sufficiently large, we have
\bb\label{7.0.1}
\phi_n'(x_0)<-\f{\|\phi_n\|_1}{\sqrt{2}},
\ee
for some $x_0\in\R$.

Fix $n$ large enough so that \eqref{7.0.1} holds and choose $u_0=\phi_n$. As a consequence of theorem \ref{teo3.4}, we know that $g_{22}$ blows up for some $x\in\R$ as long as $t$ approaches some value $T_m$ determined by the initial datum.
\end{example}

We close this section with some words about the maximal time $T_m$ of existence (lifespan) of a solution of the CH equation emanating from an initial datum in Sobolev spaces. From theorem \ref{teo3.2} we know that $u$, and the metric as well, will become unbounded before reaching a certain value determined by the initial datum. The question is: do we have any sort of information about how it is determined? An answer for this question is provided by \cite[Theorem 0.1]{dan}, see also \cite[Eq. (4.2)]{mol}, which shows a lower bound for it:
$$
T_m\geq T(u_0):=-\f{2}{\|u_0\|_1}\arctan\Big(\f{\|u_0\|_1}{\inf_{x\in\R}u_0'(x)}\Big).
$$

For the initial datum $u_0(x)=e^{-nx^2}$ considered in example \ref{example7.2}, we have
$$
T(u_0)=2\sqrt[4]{\f{2n}{\pi(n+1)^2}}\arctan\Big(\sqrt[4]{\f{\pi e^2(n+1)^2}{8n^3}}\Big).
$$

In particular, for $n\gg1$, we have
$$
T(u_0)=\sqrt{\f{2e}{n}}+O(n^{-1}).
$$

As a consequence of the quantities shown above, for the given initial datum in example \ref{example7.2} we can surely guarantee that only certain open, properly and simply connected sets contained in 
$$
{\cal S}=\R\times\Big(0,2\sqrt[4]{\f{2n}{\pi(n+1)^2}}\arctan\Big(\sqrt[4]{\f{\pi e^2(n+1)^2}{8n^3}}\Big)\Big)
$$
can be endowed with a PSS structure. 

\section{Discussion}\label{sec8}

The connection between surfaces of constant Gaussian curvature ${\cal K}=-1$ has a long history in differential geometry, dating back to the first half part of century XIX \cite[page 17]{rogers}, see also \cite[chapter 9]{cle} and \cite[chapter 1]{keti-book}. 

Roughly half a century ago, a hot topic in mathematical physics emerged after certain hydrodynamics models, more precisely, the KdV equation, was shown to have remarkable properties \cite{gardner}. In \cite{miura} there is a survey of results about the KdV equation and its importance for nourishing a new-born field whose most well known representative is just itself.

An explosion of works was seen during the 60 and 70's after \cite{gardner} exploring properties of the KdV equation, while other quite special equations were also discovered sharing certain properties with the KdV. In this context was proposed the AKNS method \cite{akns}, which reinvigorated and boosted the field emerged after the KdV, currently called {\it integrable equations} (very roughly and naively speaking, an equation sharing properties with the KdV equation). By that time, the interest on this sort of equations spred out fields, attracting people more inclined to analysis of PDEs and geometric aspects of these equations.

By the end of the 70's, \cite{sasaki} Sasaki showed an interesting connection between equations described by the AKNS method \cite{akns} and surfaces of Gaussian curvature ${\cal K}=-1$, culminating in the seminal work by Chern and Tenenblat \cite[section 1]{chern} who established the basis for what today is known as {\it PSS equations}. These works are roots for what Reyes called {\it geometric integrability}, see \cite{reyes2000,reyes2006-sel,reyes2006-jde, reyes2011}.

Equation \eqref{1.0.1} was discovered in \cite{fokas}, but became famous after its derivation as a hydrodynamic model in a paper by Camassa and Holm \cite{chprl}, and named after them, see also the review \cite{freire-cm}. Despite its physical relevance, like other integrable models physically relevant, it attracted the interests of different areas. Probably one of the most impacted was just analysis of PDEs. In particular, the works by Constantin and co-workers \cite{bressan,const1998-1,const1998-2,const2000-1,const2002-jpa} payed a crucial role, creating and developing new tools for tackling the CH equation that would later be explored not only to the CH equation itself, but also for other similar models, see \cite{dan,freire-jpa,freire-cor,freire-dpde,henry-jnmp,him-cmp,linares} to name a few. Most of these works, not to say all, deal with solutions of the CH equation with finite regularity.

Apparently, Constantin \cite{const2000-1} was the first who showed connections between the CH equation and the geometry of manifolds. However, it was not before the fundamental work by Reyes \cite{reyes2002} that it was recognised as a PSS equation. Even though these two works are concerned with the same object (the CH equation), they are completely different in nature. In fact, the results reported by Constantin \cite{const2000-1} are intrinsically related to Cauchy problems involving the CH equation, whereas those shown by Reyes are concerned with structural aspects of the equation itself, such as integrability and abstract two-dimensional surfaces.

The work by Reyes was followed by a number of works dealing with geometric aspects of CH type equations {\it à la} Chern and Tenenblat, see \cite{keti2015,tarcisio, raspajde,raspasapm,freire-tito-sam,nazime,reyes2006-sel,reyes2006-jde, reyes2011} and references therein. 

Despite the tremendous research carried out since the works \cite{bressan,const1998-1,const1998-2,const2000-1,blanco} and \cite{reyes2000,reyes2006-sel}, it is surprising that until now very few attention has been directed to geometric aspects of well-posed solutions and PSS equations. As far as I know, the first paper trying to make such a connection is \cite{nazime}, where qualitative analysis of a certain PSS equation was used to describe aspects of the corresponding metric. However, even this reference considered an analytic solution. A second attempt addressing this topic is \cite{nilay}, where Cauchy problems involving the equation considered in \cite{freire-tito-sam,nazime} were studied. In spite of the efforts made in \cite{nilay,nazime}, these works do not deal with solutions blowing up, which was first considered in \cite{freire-ch}.

In \cite{freire-ch} the notions of $C^k-$PSS and generic solutions were first considered and the blow up of metrics determined by the solutions of the CH equation were shown for two situations, depending on how the sign of the momentum behaves, see \cite[theorems 3.2 and 3.4]{freire-ch}. However, no problems related to global nature, i.e, circumstances in which the co-frame can be defined on $\R\times(0,\infty)$ or asymptotic behaviors of metrics, were considered.

The notions of generic solutions and PSS equations used in the current literature carry intrinsically $C^\infty$ regularity and this brings issues in the study of surfaces in connection with Cauchy problems. This is even more dramatic for equations like the CH because they have different representations depending on the sort of solutions one considers and they only coincide on certain Banach spaces. This explains why in \cite{freire-ch} it was needed to step forward and introduce definitions \ref{def3.4} and \ref{def3.5}.

Another important aspect of the connections between geometry and analysis made in the present paper is the condition $\omega_1\wedge\omega_2\neq0$. Whenever $\omega_1\wedge\omega_2=0$ we have \eqref{3.3.9} holding on an open set $\Omega$. This is a problem of unique continuation of solutions, whose answer would be impossible quite few time ago. 

For $c=0$, the answer for arbitrary open sets was given very recently in \cite{linares}, see also \cite{freire-jpa,freire-cor,freire-dpde}. As long as $u\big|_\Omega=0$, for some open set $\Omega$, then $u\equiv0$, see \cite[Theorem 1.3]{linares}. Our solutions emanate from a non-trivial initial datum, and then, we cannot have $u\equiv0$ on an open set $\Omega$ contained in the domain of $u$. For $c\neq0$, it is unclear if we might have $u\big|_\Omega=c$ since this unique continuation problem is still an open question, see \cite[Discussion]{freire-ch}.

The proof of Corollary \ref{cor3.1} shows that $u_x(x,t)$ vanishes at least once, for each $t$ as long as $u$ is defined, see also \cite[Theorem 2.3]{freire-ch}. As a result, the domain of $u$ cannot be wholly endowed with a PSS structure. The answer for the open question mentioned above would only clarify whether we may have open sets that cannot be endowed with a PSS structure (those in which $u$ is constant). If its answer is that $c=0$ (which I conjecture, it is the case), then Corollary \ref{cor3.1} would imply that the domain of the solution has a non-countable set of points in which we loss the structure of a PSS equation, but such a set would be of zero measure. On the other hand, if the answer to the question is that we may have $c\neq0$, then we would have a situation whose geometric implication should be better understood, but would surely imply on the existence of subsets of the domain of $u$, with positive measure, in which a PSS structure is not allowed.

Event though the ideas developed and better explored in this paper are mostly concerned with the CH equation, they can used to other PSS equations. The main point is that the techniques to deal with Cauchy problems may vary depending on the equation, and this will then impact in how to address the geometric problem. This can be seen by comparing the results established in the present paper and those in \cite{nilay,nazime}. 
Recently, the present ideas have been applied to the Degasperis-Procesi, where connections between the Cauchy problems involving that equation and PSS are investigated, see \cite{freire-dp}.

\section{Conclusion}\label{sec9}

In this paper we studied the influence of Cauchy problems and the PSS surfaces defined by the corresponding solutions. To this end, we had to propose a formulation of the notion of PSS equation and generic solutions. Our main results are reported in subsection \ref{sub2.3}, including the new already mentioned definitions. A remarkable fact reported is that any non-trivial initial datum gives rise to a PSS equation. In particular, we showed solutions breaking in finite time lead to metrics having blow up either.

\section*{Acknowledgements}

Most of this work was developed and written while I was visiting the Department of Mathematical Sciences of the Loughborough University, which I am extremely  thankful for the warm hospitality and amazing work atmosphere. I am grateful to Jenya Ferapontov, Keti Tenenblat, Artur Sergyeyev and Sasha Veselov for stimulating discussions, as well as many suggestions given. I am also indebted to Priscila Leal da Silva for her firm encouragement, support, suggestions and patience to read the first draft of this manuscript.

Last but not least, I want to thank CNPq (grant nº 310074/2021-5) and FAPESP (grants nº 2020/02055-0 and 2022/00163-6) for financial support.

\end{document}